\documentclass{amsart}
\usepackage{amssymb,amsmath, amsthm,latexsym}
\usepackage{graphics}
\usepackage{todonotes}
\usepackage{amscd}
\usepackage{graphics}
\usepackage{here}
\newcommand{\cal}[1]{\mathcal{#1}}
\theoremstyle{plain}
\newtheorem{theo}{Theorem} 
\newtheorem{cor}[theo]{Corollary}
\newtheorem{lemma}{Lemma}[section]

\newtheorem{proposition}[lemma]{Proposition}
\newtheorem{corollary}[lemma]{Corollary}
 
\theoremstyle{definition}

\newtheorem{remark}[lemma]{Remark}
\parskip=\bigskipamount

\let\phialt=\phi
\let\phi=\varphi
\let\varphi=\phialt

\begin{document}
\title{Cannon-Thurston maps and measures}
\author{Ursula Hamenst\"adt}
\date{September 26, 2026}

\begin{abstract}
  We give a short proof of the following extension of a recent result of Gadre, Maher, Pfaff and Uyanik:
  Consider a closed hyperbolic 3-manifold $M$ which fibers over the circle, with 
 fiber a closed surface $S$. 
Any measure class on the ideal boundary $S^2=\partial \mathbb{H}^3$ of hyperbolic 3-space which is invariant and totally ergodic 
under the action of 
the fundamental group of $M$ is 
singular with respect to the push-forward
by a Cannon-Thurston map of a stationary measure on $S^1$ of a random walk with finite first moment
on the fundamental group of $S$.
   \end{abstract}

\maketitle

\section{Introduction}

The fundamental group $\Gamma$ of a closed hyperbolic 3-manifold which fibers over the circle
contains a normal subgroup $\Gamma_0$ which is isomorphic to a surface group. The subgroup $\Gamma_0$
is finitely generated but exponentially distorted in $\Gamma$. In other words, the relation between
the geometry of $\Gamma_0$ and of the ambient group $\Gamma$ is very complicated.

However, the groups $\Gamma_0$ and $\Gamma$ are dynamically related in the following way.
The group $\Gamma_0$ admits an embedding as a lattice into ${\rm PSL}(2,\mathbb{R})$, and via such an
embedding, it acts on the boundary $S^1=\partial \mathbb{H}^2$ of the hyperbolic plane $\mathbb{H}^2$. Similarly, the group $\Gamma$ 
comes with a conjugacy class of an embedding into ${\rm PSL}(2,\mathbb{C})$ and hence it acts on
the boundary $S^2$ of hyperbolic 3-space $\mathbb{H}^3$ 
as a group of M\"obius transformations. By a result of
Cannon and Thurston \cite{CT07}, these two actions are related: there exists a 
continuous, surjective, finite-to-one $\Gamma_0$-equivariant map $F:S^1\to S^2$ (see also \cite{CL26} for more details on this
construction). 

While there are no $\Gamma$-invariant Radon measures on $S^1$ and $S^2$, we can ask for the 
relation between the push-forward under $F$ of 
$\Gamma_0$-invariant measure classes on $S^1$ and $\Gamma$-invariant measure classes
on $S^2$. Important examples of such invariant measure classes arise from stationary measures of random walks
on $\Gamma_0,\Gamma$, respectively. 
The main result of \cite{GMPU26} is as follows.

\begin{theo}[Theorem 4 of \cite{GMPU26}]\label{main1}
  Suppose that $\Gamma$ is the fundamental group of a closed hyperbolic 3-manifold $M$ 
  that fibers over the circle, and $\Gamma_0$ its fiber group. Then the push-forward to
  $S^2$ by a Cannon-Thurston map of the stationary measure on $S^1$ of a random walk on $\Gamma_0$ defined
  by a probability measure whose support generates $\Gamma_0$ and is of finite exponential moment 
  is singular to the stationary measure of a random walk on
  $\Gamma$ defined by a probability measure whose support generates $\Gamma$ and is of finite exponential
  moment. 
\end{theo}  

Our goal is to give a short proof of the following strengthened version of Theorem \ref{main1}. For its formulation, we consider a
$\Gamma_0$-invariant ergodic measure class $[\mu]$ on $S^1$, that is, a measure class for which the action of $\Gamma_0$ is nonsingular.
We say that $[\mu]$ 
 is \emph{induced by a current} 
if there exists a \emph{geodesic current}, that is, a $\Gamma_0$-invariant Radon measure $\nu$ on $(S^1\times S^1)\setminus \Delta$, 
which defines the measure class of the product $[\mu]\times [\mu]$. 

In the statement of the following result, we call a nonsingular ergodic action of a group $G$ on a measure space 
\emph{totally ergodic} if the action of every finite index subgroup of $G$ is still ergodic. This is equivalent to
stating that the action does not admit any finite factors.

\begin{theo}\label{main2}
  Under the assumption of Theorem \ref{main1}, the push-forward by a Cannon-Thurston map of a $\Gamma_0$-invariant 
measure class on $S^1$ induced by a current 
is singular to any $\Gamma$-invariant measure class on $S^2$ for which the $\Gamma$-action is totally ergodic.  
\end{theo}

Since a stationary 
measure on $S^1$ and $S^2$ 
of a symmetric random walk on $\Gamma_0$ and $\Gamma$ generated by a probability measure of finite first moment 
whose support generates the group 
is totally ergodic and induced by a current 
\cite{K94}, we obtain as an immediate corollary.

\begin{cor}\label{walk}
Under the assumption of Theorem \ref{main1}, the push-forward by a Cannon-Thurston map of a stationary measure  on 
$S^1$ of a symmetric random walk on $\Gamma_0$ with finite first moment is singular to a stationary  measure on $S^2$ of a 
random walk on $\Gamma$ with finite first moment. 
\end{cor}

A much more general version of Theorem \ref{main1} which applies to infinite index subgroups of groups with sufficient hyperbolicity
was established in \cite{KZ25}. Corollary 1.6 of \cite{KZ25} discusses explicitly fiber subgroups of fundamental groups of closed fibered
hyperbolic 3-manifolds. As the surface group $\Gamma_0$ is an exponentially distorted subgroup of $\Gamma$, 
the moment condition for the walk on the surface group required in \cite{KZ25} is a bit weaker than finite first moment for a word metric and hence
the result in \cite{KZ25} is a bit stronger than Corollary \ref{walk}. 
The subtlety of moment conditions is illustrated by Proposition 6.1 of \cite{KZ25} 
which shows that the measure class of a stationary measure on $S^2$ for a symmetric random walk on $\Gamma$ of finite first moment 
can in fact be realized by a stationary measure for a symmetric random walk on 
$\Gamma_0$, however the moment condition seems to be at most logarithmic.
We do not use random walks in our approach.

We also obtain the following application. 

\begin{cor}\label{nonextendable}
A $\Gamma_0$-invariant measure class $[\mu]$ on $S^1$ constructed from the Lebesgue measure on $S^2$ and 
a Cannon-Thurston map has the following properties.
\begin{enumerate}
\item The measure class $[\mu]\times [\mu]$ is ergodic under the diagonal action of $\Gamma_0$. 
\item There is no geodesic current in the measure class of $[\mu]\times [\mu]$. 
\end{enumerate} 
\end{cor}

While the proof of Theorem \ref{main1} is geometric \cite{GMPU26}, the proof of Theorem \ref{main2} 
only uses dynamical tools, and it is completely different from the approach in \cite{GMPU26} as well as from the tools used in \cite{KZ25}.

\noindent
{\bf Acknowledgement:} This note was motivated by a conversation of the author with Vaibhav Gadre during a conference at
IMPA, Rio de Janeiro, in August 2026, and it benefitted from conversations with Rafael Potrie at the same occasion. 
The author thanks IMPA for its hospitality. I am also grateful to Dongryul Kim for pointing out the article \cite{KZ25}. 
ChatGPT 5.6 was used as an editorial aid to identify 
typographical errors and, in a few places, passages whose wordings could be clarified. 
The mathematical content and all final revisions are the intellectual responsibility of the author.

\section{The circle action of ${\rm Mod}(S_{g,1})$}\label{sec:circle}

Denote by $S_{g,1}$ a surface of genus $g\geq 2$ with one marked point, and 
by $S_g$ the closed surface of genus $g$ with no marked point. Let ${\rm Mod}(S_{g,1})$ and 
${\rm Mod}(S_g)$ be the \emph{mapping class group} of $S_{g,1}$ and $S_g$, respectively, which 
can be identified with the group of automorphisms and outer automorphisms of $\pi_1(S)$.
These groups  fit into
the \emph{Birman exact sequence}
\begin{equation}\label{birman}
0\to \pi_1(S_g)\to {\rm Mod}(S_{g,1})\xrightarrow{\Pi} {\rm Mod}(S_g)\to 0.\end{equation}
This sequence identifies ${\rm Mod}(S_{g,1})$ with the automorphism group of $\pi_1(S_g)$.

The \emph{Teichm\"uller space} ${\cal T}(S_{g,1})$ of $S_{g,1}$ 
is the bundle  over the 
Teichm\"uller space ${\cal T}(S_g)$ of marked hyperbolic structures on $S_g$ whose  fiber over a marked hyperbolic  surface $X\in {\cal T}(S_g)$
is the universal covering  of $X$, which is naturally identified with the hyperbolic plane $\mathbb{H}^2$. The mapping class group
${\rm Mod}(S_{g,1})$ acts properly discontinuously
from the left on ${\cal T}(S_{g,1})$ as a group of bundle diffeomorphisms. 
The subgroup $\pi_1(S_g)$ preserves each fibers and acts cocompactly on it. Thus 
$\pi_1(S_g)\backslash {\cal T}(S_{g,1})$ is a fiber bundle over ${\cal T}(S_g)$ whose fiber over a point $X$ is just the 
hyperbolic surface $X$. In other words, this quotient is the \emph{universal curve} over ${\cal T}(S_g)$. 

The unit tangent bundle 
$T^1\mathbb{H}^2$ of $\mathbb{H}^2$ is a $\pi_1(S_g)$-invariant circle bundle on the fiber of ${\cal T}(S_{g,1})$ over $X$. 
The union of these circle bundles defines a smooth circle
bundle over ${\cal T}(S_{g,1})$ which is invariant under the action of $\pi_1(S_g)$ and descends to the 
\emph{vertical unit tangent bundle} of the fibers of the 
universal curve.
By naturality of this construction, equivalently by the Birman exact sequence, 
the circle bundle over ${\cal T}(S_{g,1})$ is in fact invariant under the whole mapping class group 
${\rm Mod}(S_{g,1})$, and it descends to the vertical tangent bundle of the universal curve over the 
moduli space of pointed hyperbolic surfaces of genus $g$.

It was discovered by Morita \cite{Mo88} that this circle bundle is \emph{flat}, that is, 
there exists a homomorphism $\rho:{\rm Mod}(S_{g,1})\to {\rm Homeo}(S^1)$ such that the circle bundle can 
be represented as ${\rm Mod}(S_{g,1})\backslash {\cal T}(S_{g,1})\times S^1$ where the action of 
${\rm Mod}(S_{g,1})$ on $S^1$ is via the homomorphism $\rho$. 
We next review the construction of the homomorphism $\rho$ following
Section 3 of \cite{H21} which takes a more geometric viewpoint than Morita's article.  

Let $D\subset \mathbb{C}$ be the unit disk, viewed as the Poincar\'e model of the hyperbolic plane $\mathbb{H}^2$. 
A marked hyperbolic surface $X\in {\cal T}(S_g)$ is defined by a discrete faithful homomorphism
$\psi_X:\pi_1(S)\to {\rm PSL}(2,\mathbb{R})$. 
Thus $X$ is the quotient of $D$ by the action of the image group $\psi_X(\pi_1(S))$.
Since ${\rm PSL}(2,\mathbb{R})$ equals the group of orientation preserving isometries of $\mathbb{H}^2$, it 
acts simply transitively on 
the unit tangent bundle $T^1\mathbb{H}^2$ of $\mathbb{H}^2$ and hence the choice of a unit tangent vector at $0\in D$ determines an 
identification of ${\rm PSL}(2,\mathbb{R})$ with $T^1\mathbb{H}^2$.

Let $x\in X$ be the projection of $0\in D$.
The group ${\rm Mod}(S_{g,1})$ can also be identified with the group of isotopy classes 
of diffeomorphisms of $S_g$ fixing the marked point $x\in X=S_g$, where the isotopy is required to fix $x$ as well. 
In this way, the fiber subgroup $\pi_1(S)$ of ${\rm Mod}(S_{g,1})$ of the Birman exact sequence (\ref{birman}),  consisting 
of inner automorphisms of $\pi_1(S)$, is identified with the \emph{point pushing group}.    

Any diffeomorphism of $S_g$ fixing $x$ is a bi-Lipschitz map for the hyperbolic structure $X$ on $S_g$ and hence 
it lifts to a $\pi_1(S_g,x)$-equivariant bi-Lipschitz map $\tilde f:\mathbb{H}^2=D\to \mathbb{H}^2=D$. Since $f$ fixes $x$, this
lift can be chosen to fix $0$.
Such a bi-Lipschitz map maps geodesic rays starting at $0$ to uniform quasi-geodesic rays, and any 
such quasi-geodesic ray is at uniformly bounded distance from a geodesic ray. As the circle $S^1=\partial D$ 
can naturally be identified with the space of geodesic rays starting at $0$, this construction then associates to a bi-Lipschitz
homeomorphism $f$ of $X$ fixing $x$ a homeomorphism $\rho(f)$ of $S^1$, and this homeomorphism only depends on the 
isotopy class of $f$ fixing $x$. Moreover, it is H\"older continuous.
This construction then defines the homomorphism
$\rho:{\rm Mod}(S_{g,1})\to {\rm Homeo}(S^1)$. Its restriction to $\pi_1(S_g)=\pi_1(X)$ is just the standard 
action of $\pi_1(X)$ on the ideal boundary of the universal covering $\mathbb{H}^2$ of $X$ (see Section 3 of \cite{H21}). 
If $\psi\in {\rm Mod}(S_{g,1})={\rm Aut}(\pi_1(S_g))$ is arbitrary, then $\psi$, viewed as an isotopy class of a diffeomorphism of $X$, 
conjugates the action of $\pi_1(X,x)$ on $S^1$ 
by the  H\"older continuous homeomorphism $\rho(\psi)$.

From now on we denote by $\Gamma_0$ the fundamental group of the hyperbolic surface $X$, with universal covering 
$\mathbb{H}^2$. Let $\psi\in {\rm Mod}(S_{g,1})$ be an element so that $\Pi(\psi)$ is a pseudo-Anosov mapping class. 
Then there is a pair $\xi_+,\xi_-$ of \emph{measured geodesic laminations} and a number $a>1$ so that 
$\psi \xi_+=a\xi_+$ and $\psi\xi_-=a^{-1}\xi_-$. This means the following. The measured lamination $\xi_+$ is a
\emph{geodesic current}, that is, a $\Gamma_0$-invariant Radon measure on the space $S^1\times S^1\setminus \Delta$
of geodesics in $\mathbb{H}^2$ 
whose support consists of \emph{simple} geodesics, equipped with a \emph{transverse measure}. 
The mapping class $\psi$ preserves this support and
multiplies the transverse measure by $a$. 

Note that ${\rm Mod}(S)$  acts on the space ${\rm Curr}$ of geodesic currents,
equipped with with the weak$^*$ topology, as a group
of transformations. 
There exists a continuous, symmetric, ${\rm Mod}(S_{g})$-invariant 
\emph{intersection form} $\iota:{\rm Curr}\times {\rm Curr}\to [0,\infty)$, see \cite{Bo88}. 

\begin{lemma}\label{positiveint}
If $\nu$ is a current defining the measure class of a product measure, then $\iota(\nu,\nu)>0$. 
\end{lemma}
\begin{proof} A geodesic current $\xi$ with $\iota(\xi,\xi)=0$ is well known to be a measured lamination. 
But a measured lamination does not define the measure class of a product. 
\end{proof}

Given a hyperbolic metric $X$ on $S_g$, we can consider the geodesic flow $\Phi^t$ on the unit tangent
bundle $T^1X$ of $X$. A $\Phi^t$-invariant finite Borel measure on $T^1X$ locally is a product of the 
Lebesgue measure on the orbits and a measure on a transversal, and this construction defines a homeomorphism
between the space of $\Phi^t$-invariant finite Borel measures on $T^1X$, equipped with the 
weak$^*$-topology, and the space of geodesic currents.  It maps ergodic invariant measures to ergodic currents. 
The classical Hopf argument shows (see also Theorem 1.3 and Theorem 1.4 of \cite{K94} which can be used as a
translation of a more general statement). 

\begin{lemma}\label{hopf}
A current $\nu$ defining the measure class of the product $[\mu]\times [\mu]$ of a $\Gamma_0$-invariant ergodic measure class
$[\mu]$ on $S^1$ is ergodic under the diagonal action of $\Gamma_0$. 
\end{lemma} 
\begin{proof}
Let $\hat \nu$ be y $\Phi^t$-invariant probability measure on $T^1X$ whose corresponding current is contained in the measure
class of $[\mu]\times [\mu]$. Then $\hat \nu$ has a local product structure with respect to the strong stable and strong 
unstable foliation. 
The classical Hopf argument therefore implies that every $\Phi^t$-invariant $\hat \nu$-integrable function is constant almost
everywhere. Hence $\hat \nu$ is ergodic, equivalently 
the current defined by $\hat \nu$ is ergodic under the diagonal action of $\Gamma_0$. 
\end{proof}

If $\nu$ is a current defining the measure class $[\mu]\times [\mu]$, then $\psi_* \nu$ is a current
defining the measure class $[\psi_* \mu]\times [\psi_* \mu]$.
By invariance, we have $\iota(\psi_* \nu,\psi_* \nu)=\iota(\nu,\nu)$. 

\begin{lemma}\label{invariant}
If the measure classes $[\mu]$ and $[\psi_* \mu]$ are not singular, then $\psi_* \nu=\nu$.
\end{lemma}
\begin{proof}
By ergodicity and $\Gamma_0$-invariance of the measure class $[\mu]$, either $[\psi_*\mu]=[\mu]$, or 
$[\mu]$ and $[\psi_*\mu]$ are singular. Thus assume that $[\mu]=[\psi_*\mu]$. Then the currents
$\nu$ and $\psi_*\nu$ are absolutely continuous and hence they define finite $\Phi^t$-invariant measures on $T^1X$ in the same
measure class. 

By Lemma \ref{hopf}, these measures are ergodic under the geodesic flow and hence 
they coincide up to a multiplicative constant, implying that $\psi_*\nu=b\nu$ for a constant $b>0$. 
But $b=1$ as $\iota(b\nu,b\nu)=b^2\iota(\nu,\nu)$ 
and $\iota(\psi_* \nu,\psi_* \nu)=\iota(\nu,\nu)>0$. 
\end{proof}

We use Lemma \ref{invariant} to show.

\begin{proposition}\label{singular10}
  Let $[\mu]$ be a measure class on $S^1$ which is invariant and ergodic under $\Gamma_0$ and induced
  by a current.  Then for any pseudo-Anosov mapping class $\psi$, the measure classes 
$[\mu]$ and $[\psi_*\mu]$ are singular.
\end{proposition}
\begin{proof}
We know that $[\psi_*\mu]$ is invariant and ergodic under the action of $\Gamma_0$. 
Thus by Lemma \ref{invariant}, if $[\mu]$ and $[\psi_*\mu]$ are
not singular, then a current $\nu$ defining the measure class of $[\mu]\times [\mu]$ coincides with $\psi_*\nu$. 

Let $\xi$ be a measured lamination so that $\psi_*\xi =a\xi$ for some $a>1$. Then we have 
\[\iota(\xi,\nu)=\iota(\psi_*\xi,\psi_*\nu)=\iota(a\xi,\nu)=a \iota(\xi,\nu)\]
and hence $\iota(\xi,\nu)=0$. Since $\psi$ is a pseudo-Anosov mapping class, the support of the measured lamination
$\xi$ fills the surface $S$, and hence $\iota(\xi,\nu)=0$ implies that $\nu$ is supported in $\xi$. 
But then $\iota(\nu,\nu)=0$, 
which contradicts Lemma \ref{positiveint}. 
\end{proof}

\section{Cannon-Thurston maps}

Consider now a closed hyperbolic 3-manifold $M$ which fibers over $S^1$, with fiber a closed surface $S_g$ of genus $g$.
The fundamental group
$\Gamma$ of $M$ fits into an exact sequence
\begin{equation}\label{birman2}
0\to \pi_1(S_g)\to \Gamma\to \mathbb{Z}\to 0.\end{equation}
The monodromy of the fibration, which is the outer automorphism of $\pi_1(S_g)$ 
determined by a generator of the quotient group $\mathbb{Z}$, 
 is a pseudo-Anosov mapping class. Thus Proposition \ref{singular10} shows. 
 
\begin{corollary}\label{singular}
If the measure class $[\mu]$ is induced by a current then 
the measure
classes $[\mu]$ and $[\rho(\psi)_*\mu]$ are singular. 
\end{corollary}

The group $\Gamma$ is hyperbolic and quasi-isometric to hyperbolic 3-space and hence its Gromov boundary
equals the two-sphere $S^2$. 
In spite of the fact that the subgroup $\Gamma_0=\pi_1(S_g)$ of $\Gamma$ is exponentially distorted, 
Cannon and Thurston showed that there is a continuous, $\Gamma_0$-equivariant, finite-to-one
map
$F:S^1\to S^2$ \cite{CT07}.  
See also \cite{CL26} for a comprehensive account.

These fibers can be understood explicitly. 
Namely, consider the equivalence relation $\sim$ on $S^1$ defined by $x\sim y$ if and only if $F(x)=F(y)$.
Following \cite{CL26}, this relation is the union of two disjoint (no perfect fits) \emph{laminar} relations
in the sense of Definition 1.16 of \cite{CL26}:
Its equivalence classes are closed subsets of $S^1$, the relation is closed in the space of unordered distinct pairs
of points, and distinct equivalence classes are \emph{unlinked}. As the relation is invariant under the action of
$\Gamma_0$, this can be thought of as follows. If $F(x)=F(y)$ with $x\not=y\in S^1$, then 
the geodesic joining $x$ to $y$ projects to a simple geodesic 
in $S_g=\Gamma_0\backslash \mathbb{H}^2$ (for any discrete embedding of $\Gamma_0$ into 
${\rm PSL}(2,\mathbb{R})$).  
Hence $F$ is injective on the Borel set $\Omega\subset S^1$  of all points which are \emph{not} endpoints of 
lifts of \emph{simple} geodesics in $S_g$. We need the following proposition.

\begin{proposition}\label{nonsimplefull}
  Let $[\mu]$ be a $\Gamma_0$-invariant ergodic measure class on $S^1$ 
induced by a current. Then $[\mu]$ gives measure zero to the set of endpoints of lifts of simple geodesics in $S_g$. 
\end{proposition}
\begin{proof} By Lemma \ref{hopf}, 
a current $\nu$ inducing $[\mu]$ is ergodic under the action of $\Gamma_0$, equivalently, the invariant measure it defines on 
$T^1S$ is ergodic under the action of the geodesic flow $\Phi^t$. 
Then 
$[\mu]$-almost every $z\in S^1$ is the endpoint of a geodesic 
whose tangent line projects to a Birkhoff regular orbit for the geodesic flow in $T^1S_g$,
equipped with the finite invariant probability measure $\hat \nu$ corresponding to (a multiple of) $\nu$. 
We have to show that such a point cannot be the endpoint of a lift of a simple geodesic on $S_g$. 

To see this assume otherwise. Then there are geodesics $\gamma_1,\gamma_2\subset \mathbb{H}^2$ with the same
endpoint $z\in S^1$ so that the tangent line $\gamma_1^\prime$ of $\gamma_1$ is Birkhoff regular for the measure $\hat \nu$, and the 
projection of the geodesic $\gamma_2$ to $S_g$ is simple. Thus, for suitable parameterizations of $\gamma_1,\gamma_2$ by arc length
and the Sasaki metric $d_S$ on $T^1\mathbb{H}^2$, it holds
\[d_S(\gamma_1^\prime(t),\gamma_2^\prime(t))\leq Ce^{-t}\]
for a constant $C>0$, and the same holds true for the projections of $\gamma_1^\prime,\gamma_2^\prime$ 
to $T^1S_g$.

Let $P:\mathbb{H}^2\to X$ be the canonical projection.
Since $P\gamma_2$ is simple, its closure  ${\mathcal L}$ in $S_g$ is a geodesic lamination. The subset 
$T^1{\mathcal L}$ 
of $T^1S_g$ of all unit tangent vectors to geodesics in ${\mathcal L}$ is a closed subset of $T^1S_g$, intersecting
the fiber of $T^1S_g$ over a point in ${\mathcal L}$ in precisely two antipodal points, and it is the closure
of the unit tangents (for both orientations) of the geodesic $P\gamma_2$. In particular, it is compact. 
As a consequence of the previous paragraph, 
for any $\alpha >0$ 
there then exists some $t(\alpha)>0$ such that whenever 
$t>t(\alpha)$,
the Sasaki distance
between $P\gamma_1^\prime(t)$ and the compact set $T^1{\mathcal L}$ is at most $\alpha$. 

Choose a small compact arc $I\subset \gamma_2$ and consider the set $A\subset T^1S_g$ of all unit tangent vectors with 
footpoint in $I$ which make an angle of at least $\pi/4$ to the tangent of $\gamma_2$. 
This set is a compact transversal for the geodesic flow which is disjoint from $T^1{\mathcal L}$. 
As a consequence, there is a number $\epsilon >0$ so that the compact set $B=\cup_{-\epsilon \leq t\leq \epsilon}\Phi^tA$ is disjoint from 
the compact set $T^1{\mathcal L}$ as well and hence has positive Sasaki distance $\alpha >0$ from $T^1{\mathcal L}$.

Now note that the invariant probability measure $\hat \nu$ 
corresponding to the current 
$\nu$ which defines the measure class of $\mu\times \mu$ gives positive measure to every open subset of $T^1S_g$. 
As $B$ contains an open subset of $T^1S_g$ and $\gamma_1^\prime(0)$ is a Birkhoff regular point
for the geodesic flow $\Phi^t$ on $T^1S_g$ with respect to  $\hat \nu$,
the forward $\Phi^t$-orbit of $(P \gamma_1)^{\prime}(t(\alpha))$, which is just the tangent line 
$(P\gamma_1)^\prime[t(\alpha),\infty)$, is dense in $T^1S_g$ 
and hence it intersects
$B$ for arbitrarily large time. This 
is a contradiction which complete the proof.
 \end{proof}

Together this yields Theorem \ref{main2} and Corollary \ref{nonextendable}.

\begin{proposition}\label{cannon}
Let $[\mu]$ be a $\Gamma_0$-invariant ergodic measure class on $S^1$ 
induced by a current. 
Then the push-forward
$F_*[\mu]$ of $[\mu]$ under a Cannon-Thurston map $F$ is singular with respect to 
any $\Gamma$-invariant totally ergodic measure class on $S^2$. 
\end{proposition}
\begin{proof}[Proof of Theorem \ref{main2}]
As a consequence of Proposition \ref{nonsimplefull}, the measure class $[\mu]$ gives full mass to the Borel set of points 
on $S^1$ on which the map $F$ is injective, and the same holds true for $[\psi_*\mu]$. As a consequence of 
Corollary \ref{singular}, the measure classes $F_*\mu$ and $F_*\psi_*\mu$ are singular. 
Thus, if $\eta$ is a $\Gamma$-invariant ergodic measure class on $S^2$ which is 
not singular with respect to $F_*[\mu]$, then the finite index subgroup of $\Gamma$ generated 
by $\Gamma_0$ and $\psi^2$ preserves a nontrivial measurable set and hence indeed, such a measure 
is not totally ergodic, that is, it admits a finite factor. 
\end{proof}

\begin{corollary}\label{noin}
The Lebesgue measure on $S^2$ defines a $\Gamma_0$-invariant measure class $[\mu]$ 
on $S^1$ so that $[\mu]\times [\mu]$ is ergodic under the diagonal action of $\Gamma_0$, but 
there is no geodesic current defining $[\mu]\times [\mu]$.
\end{corollary}
\begin{proof}
The Lebesgue measure class $[\lambda]$ on $S^2$ is well known to be totally ergodic under the action of $\Gamma$, moreover 
the measure class $[\lambda]\times [\lambda]$ is ergodic under the diagonal action of $\Gamma_0$ 
\cite{C01,H02}. 

By Proposition 5.5 of \cite{FMP26}, $[\lambda]$ 
gives full measure to the set of points in $S^2$ with only one preimage under $F$. This allows to pull the 
measure back to a measure on $S^1$ whose measure class $[F^*\lambda]$ is invariant under the action of $\Gamma_0$. 
We refer to Section 5 of \cite{FMP26} for details. By ergodicity of $[\lambda]\times [\lambda]$ under
the diagonal action of $\Gamma_0$, the same holds true 
for the measure class $[F^*\lambda]\times [F^*\lambda]$. 
However, by Proposition \ref{cannon}, $[\lambda]$  can not be induced by a current. 
 \end{proof}

\begin{remark}
Corollary \ref{noin} easily extends to measure classes on $S^2$ which are defined by Gibbs equilibrium states for the geodesic flow
on the unit tangent bundle of $M$. However, such a statement relies on Proposition 5.5 of \cite{FMP26} whose proof extends, 
but which is not stated in this generality. 
\end{remark}

\noindent
Math. Inst. Univ. Bonn, Endenicher Allee 60, 53115 Bonn, Germany\\
email: ursula@math.uni-bonn.de

\end{document}